\documentclass[11pt,a4paper]{article}

\usepackage[T1]{fontenc}
\usepackage[utf8]{inputenc}
\usepackage{lmodern}
\usepackage{amsmath,amssymb,amsthm}
\usepackage{xcolor}
\usepackage{microtype}
\usepackage[colorlinks=true,linkcolor=blue,urlcolor=blue,citecolor=blue]{hyperref}
\hypersetup{
  pdftitle={Formalizing Singer Sidon Constructions and Sidon Set Infrastructure in Lean 4},
  pdfauthor={David B. Hulak, Arthur F. Ramos, Ruy J. G. B. de Queiroz},
  pdfsubject={Formalized Mathematics, Additive Combinatorics},
  pdfkeywords={Singer theorem, Sidon sets, Lean 4, formalization, additive combinatorics}
}
\usepackage{mathrsfs}
\usepackage{enumitem}
\usepackage{booktabs}
\usepackage{graphicx}
\usepackage{array}
\usepackage{float}

\usepackage{listings}
\usepackage{upquote}
\lstdefinelanguage{Lean4}{
  morekeywords={theorem,lemma,def,noncomputable,private,set_option,
    import,namespace,end,open,variable,instance,where,by,exact,intro,
    refine,rcases,obtain,have,calc,rw,simp,linarith,omega,ring,
    apply,constructor,cases,induction,with,fun,let,set,show,
    suffices,match,if,then,else,return,do,sorry},
  sensitive=true,
  morecomment=[l]{--},
  morecomment=[n]{/-}{-/},
  morestring=[b]",
  literate={→}{$\to$}1 {←}{$\leftarrow$}1 {↦}{$\mapsto$}1
           {≤}{$\leq$}1 {≥}{$\geq$}1 {≠}{$\neq$}1 {∈}{$\in$}1
           {∉}{$\notin$}1 {⊆}{$\subseteq$}1 {⊓}{$\sqcap$}1
           {⊔}{$\sqcup$}1 {⊥}{$\bot$}1 {⊤}{$\top$}1
           {ℤ}{$\mathbb{Z}$}1 {ℕ}{$\mathbb{N}$}1 {ℝ}{$\mathbb{R}$}1
           {×}{$\times$}1 {⟨}{$\langle$}1 {⟩}{$\rangle$}1
           {∀}{$\forall$}1 {∃}{$\exists$}1 {¬}{$\lnot$}1
           {∧}{$\land$}1 {∨}{$\lor$}1 {⧸}{/}1,
}
\newtheorem{theorem}{Theorem}[section]
\newtheorem{lemma}[theorem]{Lemma}

\newtheorem{corollary}[theorem]{Corollary}
\newtheorem{definition}[theorem]{Definition}
\newtheorem{remark}[theorem]{Remark}

\newcommand{\LeanBreakTwo}[2]{%
  \begin{tabular}[t]{@{}l@{}}\texttt{#1}\\\texttt{#2}\end{tabular}}
\newcommand{\LeanBreakThree}[3]{%
  \begin{tabular}[t]{@{}l@{}}\texttt{#1}\\\texttt{#2}\\\texttt{#3}\end{tabular}}

\title{Formalizing Singer Sidon Constructions\\
and Sidon Set Infrastructure in Lean~4}

\author{David B.~Hulak\thanks{Corresponding author.}\\
\small Independent Researcher\\
\small \texttt{dbhulak@gmail.com}\\
\small \url{https://orcid.org/0009-0002-8056-1774}
\and
Arthur F.~Ramos\\
\small Microsoft\\
\small \texttt{arfreita@microsoft.com}\\
\small \url{https://orcid.org/0009-0003-3568-0325}
\and
Ruy J.~G.~B.~de~Queiroz\\
\small Centro de Inform\'atica, Universidade Federal de Pernambuco\\
\small \texttt{ruy@cin.ufpe.br}\\
\small \url{https://orcid.org/0000-0003-1482-0977}}

\date{}

\begin{document}

\maketitle

\begin{abstract}
Erd\H{o}s Problem~30 asks for sharp asymptotics of the Sidon extremal
function~$h(N)$, and Singer's construction is the classical source of
lower-bound examples matching the main term.  We present a Lean~4 formalization of Singer's
Sidon set construction, together with reusable Sidon-set
infrastructure for additive combinatorics.  For every prime power~$q=p^k$, we prove the existence of a Sidon set
modulo~$q^2+q+1$ of cardinality~$q+1$; the prime-field case \(q=p\) is the
base presentation.  The proof proceeds through a
non-trivial algebraic chain:
construction of the base field and its degree-three extension, analysis of the trace
kernel as a 2-dimensional subspace over the base field, a geometric argument via subspace
intersections establishing the multiplicative Sidon property in the quotient
group, and a transfer from quotient multiplication to modular integer addition.
Around this central result, we develop a reusable Sidon set
library for additive combinatorics.  It comprises interval Sidon sets, modular
Sidon sets, the extremal function~$h(N)$, Lindstr\"om's cross-difference
inequality, a Johnson-route shift-incidence upper bound of the form
$h(N) \le \sqrt N + N^{1/4} + O(1)$, exact
representation-function identities, and unconditional two-sided
\(h(N)=\Theta(\sqrt N)\) bounds with exact floor-rounded finite statements for
\(N \ge 5\).  We further formalize a conditional reduction:
subpolynomial prime gaps together with a full subpolynomial upper-error
hypothesis for \(h(N)\) imply the Erd\H{o}s Problem~30 estimate
$h(N) = \sqrt{N} + O_\varepsilon(N^\varepsilon)$ for every
$\varepsilon>0$.  The Singer/Sidon modules and transfer lemmas
comprise 7\,541 lines of Lean~4 with zero active uses of \texttt{sorry}.
We describe the
mathematical lessons learned, focusing on how formalization clarifies the precise
scope of classical arguments and forces explicit treatment of the
passage from the field-theoretic construction to integer Sidon predicates.
\end{abstract}


\section{Introduction}
\label{sec:intro}

\subsection{Erd\H{o}s Problem 30}

Let $h(N) = \max\{|A| : A \subseteq \{1,\dots,N\} \text{ is Sidon}\}$, where a
set~$A$ of integers is \emph{Sidon} (or a $B_2$~set) if all pairwise sums
$a + b$ with $a \le b$ (including $a=b$) and $a, b \in A$ are distinct.
Lindstr\"om~\cite{Lindstrom1969}
established the upper bound
\[
  h(N) \le \sqrt{N} + N^{1/4} + 1,
\]
sharpening the earlier Erd\H{o}s--Tur\'an~\cite{ET1941} bound.  Later
work~\cite{BaloghFurediRoy2023,CarterHunterOBryant2025} improves the constant
in the \(N^{1/4}\) term but not the exponent.  On the lower side, Singer's
construction and prime-finding results give
\(h(N)\ge \sqrt N-O(N^{\theta_{\mathrm{gap}}})\).  Thus an upper error
\(O(N^{\theta_{\mathrm{up}}})\) and this Singer--prime lower error combine to
give \(h(N)=\sqrt N+O(N^{\max(\theta_{\mathrm{up}},\theta_{\mathrm{gap}})})\).
Here \(\theta_{\mathrm{gap}}\) denotes the prime-gap error exponent and
\(\theta_{\mathrm{up}}\) the upper-bound error exponent for \(h(N)\).
Erd\H{o}s Problem~30~\cite{Erdos1957,Guy2004} asks whether this can be improved to
$h(N) = \sqrt{N} + O_\varepsilon(N^\varepsilon)$
for every $\varepsilon > 0$.  This remains open.

\subsection{Why formalize?}

Formalization of classical combinatorial number theory serves three purposes
beyond mere verification:
\begin{enumerate}[nosep]
  \item \textbf{Scope discipline}: the type system forces the formalizer to
    confront exactly which hypotheses each theorem requires.  For Singer's
    construction, the development separates the prime-field presentation from
    the prime-power theorem, where the base field is \texttt{GaloisField p k}
    and the degree-three extension is supplied by Mathlib's finite-field
    extension API.
  \item \textbf{Field-to-integer transfer}: the passage from algebraic structures
    (quotient groups, field extensions) to combinatorial predicates (modular
    Sidon, interval containment) requires explicit definitions and transfer
    lemmas that informal proofs often elide.
  \item \textbf{Conditional dependencies}: formal verification makes the logical
    dependencies between open conjectures (prime gaps, upper bounds) and the
    target statement (Erd\H{o}s 30) machine-checkable.
\end{enumerate}

\subsection{Contributions}

Our main contributions are:
\begin{enumerate}[nosep]
  \item A Singer formalization presented through explicit
    Sidon-set predicates ($\S$\ref{sec:singer}), with the algebraic chain
    (trace-kernel proof, no-invariant-subspace argument, Grassmann
    intersection, projective-line representatives) and the transfer through
    the modular Sidon predicate fully unfolded; this yields the prime-field
    theorem and the prime-power theorem giving Sidon sets of size~$q+1$ modulo
    \(q^2+q+1\) for every prime power \(q=p^k\).
  \item A reusable Sidon set library ($\S$\ref{sec:sidon-infra}) with interval
    and modular variants, the extremal function~$h(N)$, and modular-to-interval
    transfer machinery.
  \item Unconditional bounds on~$h(N)$ ($\S$\ref{sec:bounds}), including
    $h(N) = \Theta(\sqrt{N})$ and a formalized Johnson-route upper bound
    $h(N) \le \lfloor\sqrt{N}\rfloor+\lfloor\sqrt{\lfloor\sqrt{N}\rfloor}\rfloor+2$ for $N\ge16$,
    matching Lindstr\"om's $\sqrt N + N^{1/4}+O(1)$ leading order.
    This yields two separate consequences: the one-sided upper-bound half of
    Erd\H{o}s~30 for $\varepsilon\ge1/4$, and, from the coarser two-sided
    bounds, the full two-sided estimate for the coarse range
    $\varepsilon \ge 1/2$.
  \item A conditional reduction ($\S$\ref{sec:conditional}) whose mathematical
    content is the prime-gap-driven lower-bound transfer
    \texttt{sidonLowerBound\_\allowbreak{}of\_\allowbreak{}singer\_\allowbreak{}and\_\allowbreak{}gap},
    including the threshold arithmetic that turns subpolynomial gaps into a
    Sidon lower error of the same exponent; the final assembly
    \texttt{erdos30\_\allowbreak{}of\_\allowbreak{}gap\_\allowbreak{}and\_\allowbreak{}upper}
    packages this with the upper-bound hypothesis as a clean corollary.
  \item Mathematical lessons ($\S$\ref{sec:lessons}) on how formalization
    reshapes the classical Singer argument and exposes hidden dependencies in
    the field-to-integer transfer.
\end{enumerate}

\subsection{Scope and limitations}

We separate deliberate modeling choices from practical limitations.

\smallskip\noindent\textbf{Modeling choices} (deliberate scope decisions that do not weaken
stated results):
\begin{itemize}[nosep]
  \item Singer's original theorem~\cite{Singer1938} applies to all prime
    powers~$q$.  We now formalize both the prime case \(q=p\), using Lean's
    \texttt{GaloisField p 3} type, and the prime-power case \(q=p^k\), using
    \texttt{GaloisField p k} as the base field and a degree-three finite-field
    extension over it.
  \item The conditional Erd\H{o}s~30 reduction relies on a subpolynomial prime
    gap hypothesis (implied by Cram\'er's conjecture~\cite{Cramer1936} but weaker) which
    remains open.  This is inherent: no unconditional proof of Erd\H{o}s~30
    is known.
\end{itemize}

\smallskip\noindent\textbf{Practical limitations} (results not provided by this paper):
\begin{itemize}[nosep]
  \item The prime-power Singer theorem does not close Erd\H{o}s~30.  It
    discharges the algebraic Singer-family input at prime powers, but the
    subpolynomial prime-gap and upper-bound hypotheses in the full conditional
    reduction remain open.
\end{itemize}

\subsection{Development overview}

The files listed in Table~\ref{tab:modules} comprise 7\,541 lines of
Lean~4.  The build uses toolchain \texttt{leanprover/lean4:v4.29.0}
and the Mathlib commit
\begin{center}
\texttt{8a178386ffc0f5fef0b77738bb5449d50efeea95}.
\end{center}
This Mathlib commit bumps the Mathlib toolchain to Lean v4.29.0 and, together with the
toolchain pin, fully determines the build environment.  Those files sit inside a larger Erd\H{o}s Problem~30 repository,
which also contains boundary-analysis modules, computational verification
modules (using \texttt{native\_decide}), and experimental proof paths.  The
listed Singer/Sidon files reported in this paper use no \texttt{sorry} and no
\texttt{native\_decide}.  For the declarations listed in
\texttt{AxiomCheck.lean}, the axiom audit reports only Lean's standard
foundations (\texttt{propext}, \texttt{Quot.sound},
\texttt{Classical.choice}).  The use of classical logic is inherent: the
extremal function~$h(N)$ is defined via a non-constructive maximum, and the
Galois field construction requires classical existence of irreducible
polynomials.

\begin{table}[ht]
\centering
\caption{Modules listed for the Singer/Sidon formalization.}
\label{tab:modules}
\small
\resizebox{\textwidth}{!}{%
\begin{tabular}{@{}llr@{}}
\toprule
Module & Content & Lines \\
\midrule
\texttt{Sidon.lean} & Sidon predicate, basic properties & 85 \\
\texttt{Interval.lean} & Interval/modular Sidon, transfer & 526 \\
\texttt{CyclicWindow.lean} & Cyclic-window relabeling & 177 \\
\texttt{Averaging.lean} & Window averaging lemma & 183 \\
\texttt{GapBound.lean} & Distinct-gap lower bound & 192 \\
\texttt{SidonGap.lean} & Sidon gap distinctness and quantitative transfer & 1395 \\
\texttt{PrimePowerTransfer.lean} & Singer/Bose--Chowla parameter wrappers & 420 \\
\texttt{PrimeGapTransfer.lean} & Bertrand and prime-gap transfer & 536 \\
\texttt{FormalStatement.lean} & $h(N)$, Erd\H{o}s~30 statement & 171 \\
\texttt{Singer.lean} & Trace kernel, no invariant subspace & 149 \\
\texttt{SingerBridge.lean} & Scaled submodules, index & 159 \\
\texttt{SingerSidon.lean} & Quotient-group Sidon property & 129 \\
\texttt{SingerTheorem.lean} & Quotient-to-modular transfer & 316 \\
\texttt{SingerPrimePowerCore.lean} & Prime-power finite-field setup & 67 \\
\texttt{SingerPrimePowerAlgebra.lean} & Prime-power trace-kernel algebra & 118 \\
\texttt{SingerPrimePowerBridge.lean} & Prime-power quotient bridge & 112 \\
\texttt{SingerPrimePowerUnits.lean} & Prime-power unit quotient & 95 \\
\texttt{SingerPrimePowerSidon.lean} & Prime-power quotient Sidon property & 143 \\
\texttt{SingerPrimePowerTheorem.lean} & Prime-power theorem discharge & 349 \\
\texttt{SingerPrimePowerSmokeTest.lean} & \(q=4\) smoke test & 25 \\
\texttt{PrimePowerLowerBound.lean} & Prime-power lower-bound reduction & 250 \\
\texttt{Lindstrom.lean} & Cross-difference inequality & 208 \\
\texttt{LindstromImproved.lean} & Johnson/Lindstr\"om upper bound & 348 \\
\texttt{RepresentationFunction.lean} & Pair-sum representation bound & 38 \\
\texttt{RepresentationL2.lean} & L$^1$/L$^2$ representation identities & 149 \\
\texttt{AdditiveEnergy.lean} & Exact Sidon additive energy & 152 \\
\texttt{DifferenceSet.lean} & Exact difference-set cardinality & 88 \\
\texttt{SumsetCard.lean} & Exact sumset cardinality & 145 \\
\texttt{SidonCharacterization.lean} & Sumset-cardinality characterization & 72 \\
\texttt{SumDiffComparison.lean} & Sum/difference comparison & 68 \\
\texttt{UnconditionalBounds.lean} & Two-sided bounds, partial E30 & 355 \\
\texttt{ConditionalLowerBound.lean} & Prime gap reduction & 205 \\
\texttt{ConditionalErdos30.lean} & Discharged conditional E30 & 22 \\
\texttt{FinalTheoremSurface.lean} & Assembly & 94 \\
\midrule
\textbf{Total (listed files)} & & $7\,541$ \\
\bottomrule
\end{tabular}
}
\end{table}

The line counts in Table~\ref{tab:modules} are physical source lines
(including comments and blank lines), computed at the cited artifact commit by
running \texttt{wc -l} on exactly the listed files.  The table is a
manuscript-facing listed-file inventory, not a transitive-import closure; for example,
\texttt{FinalTheoremSurface.lean} also imports a boundary-certificate module
used for repository-status bookkeeping, outside the mathematical dependency
chain summarized here.

\clearpage

\section{Mathematical Background}
\label{sec:background}

\subsection{Sidon sets}

A set $A \subseteq \mathbb{Z}$ is \emph{Sidon} if whenever $a + b = c + d$ with
$a, b, c, d \in A$ and $a \le b$, $c \le d$, then $\{a,b\} = \{c,d\}$.
Equivalently, one may omit the ordering convention and allow the two possible
ordered identifications; the Lean predicate uses this unordered collision form.
For a positive integer~$M$, we say $A \subseteq \mathbb{Z}$ is \emph{Sidon modulo~$M$} if
$M \mid (a + b) - (c + d)$ with $a,b,c,d \in A$ implies
$\{a,b\} = \{c,d\}$ as unordered pairs of integers.  Modular Sidon immediately implies
ordinary Sidon: if $a + b = c + d$ as integers, then
$M \mid (a+b)-(c+d) = 0$, so the modular condition gives
$\{a,b\} = \{c,d\}$.  The converse requires a bounded embedding:

\begin{lemma}[No-wraparound, \texttt{IsSidon.isSidonMod\_of\_interval}]
If $A$ is Sidon with $A \subseteq \{1,\dots,N\}$ and $M \ge 2N - 1$, then
$A$ is Sidon modulo~$M$.
\end{lemma}

\noindent The bound ensures all pairwise-sum differences lie in
$\{-(2N-2),\dots,2N-2\} \subset (-M,M)$, so $M$-divisibility forces the
difference to be zero, reducing to an ordinary collision.

\subsection{Singer's construction}

A \emph{$(v,k,\lambda)$-difference set} is a subset $D$ of
$\mathbb{Z}/v\mathbb{Z}$ with $|D| = k$ such that every nonzero residue is
representable as exactly~$\lambda$ ordered differences $d_i - d_j$; a $(v,k,1)$-set
is called \emph{perfect}.  The $\lambda = 1$ condition implies the Sidon
property in the cyclic group (at most one representation per difference);
for Singer parameters, where $k(k-1) = v-1$, the converse also holds.

Singer~\cite{Singer1938} proved that for every prime power~$q$, the quotient
$\mathrm{GF}(q^3)^\times / \mathrm{GF}(q)^\times$ is cyclic of order
$q^2 + q + 1$, and the image of the nonzero trace-kernel elements
$\ker(\mathrm{Tr}_{\mathrm{GF}(q^3)/\mathrm{GF}(q)}) \setminus \{0\}$
under the quotient map forms a $(q^2+q+1, q+1, 1)$
perfect difference set of size $q + 1$ in $\mathbb{Z}/(q^2+q+1)\mathbb{Z}$.
Indeed, the trace kernel has \(q^2\) elements, so its \(q^2-1\) nonzero
elements fall into \((q^2-1)/(q-1)=q+1\) scalar orbits.
The formalization now covers this prime-power scope: the prime presentation uses
\(\mathrm{GF}(p^3)\) over \(\mathrm{GF}(p)\), while the prime-power theorem uses
\(\texttt{GaloisField p k}\) as the base field and a degree-three finite-field
extension over that base.

The Sidon property rests on a geometric fact.  Let
$V = \ker(\mathrm{Tr}_{\mathrm{GF}(q^3)/\mathrm{GF}(q)})$, a 2-dimensional
$\mathrm{GF}(q)$-subspace (by rank-nullity applied to the surjective trace).
For an element \(\alpha\) of the extension not lying in the embedded base field:
\begin{enumerate}[nosep]
  \item The scaled subspace $\alpha V$ differs from~$V$: if $\alpha V = V$,
    then $V$ would be a proper $\alpha$-invariant submodule of
    $\mathrm{GF}(q^3)$, contradicting the no-invariant-subspace theorem.
  \item Two distinct 2-dimensional subspaces of a 3-dimensional space intersect
    in dimension~1 (Grassmann formula).
  \item Therefore $V \cap \alpha^{-1}V$ is 1-dimensional for such
    \(\alpha\in \mathrm{GF}(q^3)\setminus\mathrm{GF}(q)\).
\end{enumerate}
This forces any ``product collision'' in the quotient (i.e., $[u]\cdot[v] =
[w]\cdot[x]$) to degenerate into a pair equality, establishing the
multiplicative Sidon property that is then transferred to additive modular
arithmetic via a cyclic quotient group isomorphism (the abstract
discrete-logarithm map).

\subsection{Known bounds on \texorpdfstring{$h(N)$}{h(N)}}

The pair-difference counting argument gives $h(N) \le \sqrt{2N} + 1$.
Lindstr\"om's cross-difference inequality~\cite{Lindstrom1969} sharpens this to
$h(N) \le \sqrt{N} + N^{1/4} + 1$.  Balogh--F\"uredi--Roy~\cite{BaloghFurediRoy2023}
and Carter--Hunter--O'Bryant~\cite{CarterHunterOBryant2025} subsequently
improved the constant in the \(N^{1/4}\) term, the latter giving
\(h(N)\le N^{1/2}+0.98183N^{1/4}+O(1)\); these refinements do not change the
exponent-level reduction formalized here.  For a broader annotated guide to
the Sidon-sequence literature, see O'Bryant~\cite{OBryant2004}.  For the lower
bound, Singer sets transferred to intervals via prime-finding results give
$h(N) \ge \sqrt{N} - O(N^\theta)$ where $\theta$ depends on the available
prime gap theorem (Bertrand gives $\theta = 1/2$;
Baker--Harman--Pintz~\cite{BHP2001} gives $\theta = 0.2625$).
Bose--Chowla~\cite{BoseChowla1962} is another classical algebraic source of
finite Sidon sets; in the present development it appears at the
parameter-wrapper and transfer-lemma level, with Singer's theorem supplying
the construction theorem used for the formal lower-bound derivation.

\section{The Singer Formalization}
\label{sec:singer}

\subsection{Trace-kernel algebra over finite-field extensions}

The algebraic foundation consists of three theorems about a degree-three finite
extension viewed as a 3-dimensional vector space over its base field.  In the
prime-field presentation this is \(\mathrm{GF}(p^3)\) over \(\mathrm{GF}(p)\);
in the prime-power theorem it is the extension \(L\) of
\(\texttt{GaloisField p k}\) supplied by \texttt{FiniteField.Extension}.
Throughout the displayed prime-field declarations below \(p\) is prime; the
prime-power declarations add a positive exponent \(k\).

\begin{theorem}[Trace kernel dimension]
\label{thm:finrank-ker}
$\dim_{\mathrm{GF}(p)} \ker(\mathrm{Tr}_{\mathrm{GF}(p^3)/\mathrm{GF}(p)}) = 2$.
\end{theorem}

\begin{lstlisting}
theorem finrank_ker_trace :
    finrank (ZMod p) (Algebra.trace (ZMod p) (GaloisField p 3)).ker = 2
\end{lstlisting}

\begin{theorem}[No proper invariant subspace]
\label{thm:no-invariant}
For $\alpha \notin \mathrm{GF}(p)$, multiplication by~$\alpha$ preserves no
proper non-trivial $\mathrm{GF}(p)$-submodule of\/~$\mathrm{GF}(p^3)$.
\end{theorem}

The proof constructs the three linearly independent vectors
$\{v, \alpha v, \alpha^2 v\}$ for any $v \ne 0$, using the fact that
$\alpha$ has degree-3 minimal polynomial over~$\mathrm{GF}(p)$ (since
$[\mathrm{GF}(p^3):\mathrm{GF}(p)] = 3$ is prime, there is no intermediate
field, so every $\alpha \notin \mathrm{GF}(p)$ generates the full extension).
If $c_0 v + c_1 \alpha v + c_2 \alpha^2 v = 0$ with $c_i \in \mathrm{GF}(p)$
not all zero, then $v(c_0 + c_1\alpha + c_2\alpha^2) = 0$; since $v \ne 0$,
this gives a nontrivial degree-$\le 2$ relation for~$\alpha$
over $\mathrm{GF}(p)$, contradicting the fact that the minimal polynomial
of~$\alpha$ has degree~$3$.
If a proper submodule~$V$ with $\dim V < 3$ were $\alpha$-invariant, it
would contain these three independent vectors, forcing $\dim V \ge 3$---a
contradiction.

\begin{theorem}[Intersection dimension]
\label{thm:intersection}
Any two distinct 2-dimensional $\mathrm{GF}(p)$-subspaces
of\/~$\mathrm{GF}(p^3)$ intersect in dimension~1.
\end{theorem}

This follows from the Grassmann formula
$\dim(V + W) + \dim(V \cap W) = \dim V + \dim W$.  Since $V \ne W$ are both
2-dimensional, $\dim(V + W) = 3$, hence $\dim(V \cap W) = 2 + 2 - 3 = 1$.

\subsection{Quotient-group Sidon property}

We define the canonical group homomorphism
$\mathtt{singerMk} : \mathrm{GF}(p^3)^\times \to
 \mathrm{GF}(p^3)^\times / \mathrm{GF}(p)^\times$
and write $[u]$ for the coset $u \cdot \mathrm{GF}(p)^\times$.
Multiplication below denotes field multiplication in $\mathrm{GF}(p^3)$.
We prove:

\begin{theorem}[Singer quotient Sidon]
\label{thm:singer-sidon}
Let $u, v, w, x \in \ker(\mathrm{Tr}) \setminus \{0\}$ and $\alpha \in
\mathrm{GF}(p)^\times$ satisfy $u \cdot v = \alpha \cdot (w \cdot x)$.  Then
either $[u] = [w]$ and $[v] = [x]$, or $[u] = [x]$ and $[v] = [w]$.
\end{theorem}

Here \(\mathrm{GF}(p)^\times\) is viewed as the embedded scalar subgroup of
\(\mathrm{GF}(p^3)^\times\).  Equivalently, equality
\([u][v]=[w][x]\) in the quotient group is the assertion that
\(u\cdot v=\alpha\cdot(w\cdot x)\) for some
\(\alpha\in\mathrm{GF}(p)^\times\); the displayed theorem is this scalar form.

The proof splits on whether $\beta = u \cdot w^{-1}$ lies in~$\mathrm{GF}(p)$:
\begin{itemize}[nosep]
  \item \textbf{Base case} ($\beta \in \mathrm{GF}(p)$): then $[u] = [w]$
    directly from the coset characterization $[a] = [b] \iff a^{-1}b \in
    \mathrm{GF}(p)^\times$.  The multiplicative relation then forces $[v] = [x]$.
  \item \textbf{Non-base case} ($\beta \notin \mathrm{GF}(p)$): the
    intersection $V \cap \beta^{-1}V$ is 1-dimensional
    (Theorems~\ref{thm:no-invariant} and~\ref{thm:intersection}).
    Now $w \in V$ (given) and $\beta w = u \in V$, so $w \in \beta^{-1}V$;
    similarly $v \in V$ and from $u \cdot v = \alpha \cdot w \cdot x$ we
    compute, using \(w\ne0\) and commutativity,
    $\beta v = (u/w) \cdot v = (u \cdot v)/w = \alpha \cdot x \in V$
    (since $\alpha \in \mathrm{GF}(p)^\times$ and $x \in V$, the
    $\mathrm{GF}(p)$-subspace $V$ is closed under scalar multiplication).
    Thus $v \in \beta^{-1}V$.  Hence both~$w$ and~$v$
    lie in $V \cap \beta^{-1}V$, a 1-dimensional space, giving $v = c \cdot w$
    for some $c \in \mathrm{GF}(p)^\times$, hence $[v] = [w]$ and $[u] = [x]$.
\end{itemize}

\subsection{From quotient multiplication to modular sums}

The final step connects the quotient-group Sidon property to the combinatorial
\texttt{IsSidonMod} predicate.  This transfer has four components:

\begin{enumerate}[nosep]
  \item \textbf{Projective representatives}: we construct $p + 1$ nonzero
    elements of $\ker(\mathrm{Tr})$ representing the projective lines of the
    2-dimensional kernel, parametrized by $\mathrm{Option}(\mathrm{GF}(p))$
    (the $p$ affine lines plus one line at infinity).
  \item \textbf{Injectivity}: these representatives map to distinct cosets in
    the quotient, proved by showing that proportionality of representatives
    forces equality of the parametrizing indices.
  \item \textbf{Cyclic isomorphism}: we exhibit a group isomorphism
    $\mathbb{Z}/(p^2+p+1)\mathbb{Z} \cong \mathrm{GF}(p^3)^\times /
    \mathrm{GF}(p)^\times$ by computing the index
    $[\mathrm{GF}(p^3)^\times : \mathrm{GF}(p)^\times] =
    (p^3-1)/(p-1) = p^2+p+1$ and invoking cyclicity of the quotient
    (since $\mathrm{GF}(p^3)^\times$ is cyclic as the multiplicative group of
    a finite field, its quotient by any subgroup is cyclic).
  \item \textbf{Additive transfer}: let $g$ be a generator of the cyclic
    quotient of order $M = p^2+p+1$, and let
    $i : \mathrm{GF}(p^3)^\times/\mathrm{GF}(p)^\times \to
    \mathbb{Z}/(p^2+p+1)\mathbb{Z}$ be the group isomorphism sending
    $[x] \mapsto \log_g [x]$.  No canonical generator is used; any chosen
    cyclic generator yields such an isomorphism.  Since $i$ is a homomorphism,
    a multiplicative collision $[u][v] = [w][x]$ gives
    $i([u])+i([v]) \equiv i([w])+i([x]) \pmod{M}$,
    so the multiplicative Sidon property yields modular-additive Sidon:
    \begin{align*}
      &M \mid (i([u]){+}i([v]))-(i([w]){+}i([x])) \\
      &\quad\implies \{[u],[v]\}=\{[w],[x]\} \\
      &\quad\implies \{i([u]),i([v])\}=\{i([w]),i([x])\}.
    \end{align*}
    This is formalized via
    $\texttt{ZMod.intCast\_zmod\_eq\_zero\_iff\_dvd}$.
\end{enumerate}

Concretely, after choosing the isomorphism~$i$, the integer set~$S$ is formed
by taking the canonical integer representative in $\{0,\dots,M-1\}$ of each
residue $i([u])$, where~$u$ ranges over the chosen projective representatives
of the trace kernel.  Injectivity of the projective representatives and of the
cyclic isomorphism gives $|S|=p+1$.  The passage from equality in
$\mathbb{Z}/M\mathbb{Z}$ to the divisibility predicate in
\texttt{IsSidonMod} is exactly the equivalence
$\overline{x}=\overline{y}\iff M\mid x-y$; unordered-pair equality is then
transported back through the injective representative map.  The quotient-group
construction is therefore realized as a finite subset of~$\mathbb{Z}$ with
integer representatives.

\begin{definition}[\texttt{SingerFamilyHypothesis}]
\label{def:singer-family}
For every prime~$p$, there exists $S \subseteq \mathbb{Z}$ with $|S| = p + 1$
such that $S$ is Sidon modulo $p^2 + p + 1$.
\end{definition}

\noindent We package this as a named hypothesis so that
$\S$\ref{sec:conditional} can state conditional theorems parametrically,
remaining usable with alternative Singer-type constructions (e.g., for prime
powers):

\begin{theorem}[Singer's theorem, prime-field case]
\label{thm:singer-main}
For every prime~$p$, there exists $S \subseteq \mathbb{Z}$ with $|S| = p + 1$
that is Sidon modulo $p^2 + p + 1$.
In particular, \texttt{SingerFamilyHypothesis} holds.
\end{theorem}

\begin{lstlisting}
theorem singerFamilyHypothesis_holds : SingerFamilyHypothesis := by
  intro p hp'
  haveI : Fact (Nat.Prime p) := ⟨hp'⟩
  exact Singer.singer_sidon_set p hp'
\end{lstlisting}

\begin{remark}[Scope]
The prime-field theorem remains useful as the simplest presentation of the
algebraic chain.  The prime-power theorem reuses the same projective-line
argument over \texttt{GaloisField p k} and a degree-three finite-field
extension, and discharges the prime-power Singer family hypothesis.  It does
not discharge the subpolynomial prime-gap hypothesis or the full
Erd\H{o}s~30 statement.
\end{remark}

\begin{theorem}[Singer's theorem, prime-power case]
\label{thm:singer-prime-power}
For every prime power \(q=p^k\), there exists \(S\subseteq\mathbb{Z}\) with
\(|S|=q+1\) that is Sidon modulo \(q^2+q+1\).  In particular,
\texttt{SingerPrimePowerFamilyHypothesis} holds.
\end{theorem}

\begin{lstlisting}
theorem singerPrimePowerFamilyHypothesis_holds :
    SingerPrimePowerFamilyHypothesis := by
  intro q hq
  obtain ⟨p, k, hp_prime, hk_pos, hpk⟩ := hq
  ...
\end{lstlisting}

\noindent The theorem surface is also smoke-tested at \(q=4\), outside the
prime-only case:
\begin{lstlisting}
theorem SingerPrimePower.singer_primePower_smoke_q4 :
    ∃ S : Finset ℤ,
      IsSidonMod (4 * 4 + 4 + 1 : ℤ) S ∧ S.card = 4 + 1
\end{lstlisting}

\section{Sidon Set Infrastructure}
\label{sec:sidon-infra}

\subsection{Predicates}

We define three levels of Sidon property, each building on the previous:
\begin{itemize}[nosep]
  \item $\mathtt{IsSidon}(A)$: the set $A \subseteq \mathbb{Z}$ is Sidon.
  \item $\mathtt{IsSidonMod}(M, A)$: $A$ is Sidon modulo~$M$, i.e.,
    $M \mid (a+b)-(c+d)$ with $a,b,c,d \in A$ implies $\{a,b\}=\{c,d\}$.
  \item $\mathtt{IsIntervalSidon}(N, A)$: $A$ is Sidon and
    $A \subseteq \{1,\dots,N\}$.
\end{itemize}

The modular predicate mediates between algebraic constructions (which produce
Sidon sets in cyclic groups) and combinatorial targets (which require Sidon
sets in integer intervals).

\subsection{Cardinality identities for Sidon sets}

For finite integer Sidon sets, the library also proves exact cardinality
consequences that are independent of the Singer construction.  The theorem
\texttt{IsSidon.card\_add} gives the sharp sumset formula
$|A+A|=|A|(|A|+1)/2$, and
\texttt{isSidon\_iff\_card\_add} proves the converse characterization:
maximal sumset cardinality forces the Sidon property.  For nonempty Sidon
sets, \texttt{IsSidon.card\_sub} gives
$|A-A|=|A|^2-|A|+1$.  Separately, \texttt{IsSidon.addEnergy\_eq} gives
additive energy $2|A|^2-|A|$ without a nonemptiness hypothesis.  Combining
the exact sumset and difference-set formulas, \texttt{IsSidon.card\_sub\_ge\_card\_add}
proves $|A-A|\ge |A+A|$ for nonempty Sidon sets (strictly for $|A|\ge 3$ via
\texttt{IsSidon.card\_sub\_gt\_card\_add}).

\subsection{Representation functions}

The representation-function modules connect the Sidon predicates to Mathlib's
additive convolution API.
\begin{definition}[Representation function]
For a finite integer set~$A$, write
\[
  r_A(n)=|\{(a,b)\in A^2 : a+b=n\}|.
\]
\end{definition}
The theorem \texttt{IsSidon.repr\_le\_two} proves that, if $A$ is Sidon, then
$r_A(n)\le2$ for every integer~$n$.  The library also proves the exact
identities
\[
  \sum_{n\in A+A} r_A(n)=|A|^2,\qquad
  \sum_{n\in A+A} r_A(n)^2=2|A|^2-|A|,
\]
and the deficiency formula
\[
  \sum_{n\in A+A} (2-r_A(n))=|A|.
\]
These results are independent of Singer's construction and provide a reusable
link between Sidon-set combinatorics, additive energy, and analytic
upper-bound arguments.

\subsection{The extremal function}

\begin{definition}
$\mathtt{sidonMaximum}(N) = \max\{|A| : \mathtt{IsIntervalSidon}(N, A)\}$.
\end{definition}

\noindent We write $h(N)$ for $\mathtt{sidonMaximum}(N)$, matching the standard
notation of $\S$\ref{sec:intro}.

Existence of the maximum is proved via finiteness of the powerset of
$\{1,\dots,N\}$ and non-emptiness (the singleton $\{1\}$ is Sidon for
$N \ge 1$).  The Lean definition then uses \texttt{Classical.choose} to select
the unique maximum witness, so \texttt{sidonMaximum} is noncomputable even
though the underlying finite maximum exists.  We also prove basic properties:
positivity ($h(N) \ge 1$
for $N \ge 1$) and monotonicity ($N \le M \implies h(N) \le h(M)$).

\subsection{Transfer machinery}

The modular-to-interval transfer proceeds via cyclic window relabeling.
Given a modular Sidon set~$S$ with $\mathtt{IsSidonMod}\ M\ S$, a cyclic
window of length~$N$ starting at offset~$u$ restricts~$S$ to elements
$x$ with $(x - u) \bmod M < N$ and relabels them to
$\{1,\dots,N\}$.  The key theorem (\texttt{IsSidonMod.windowSidon}):

\begin{lemma}[Cyclic window transfer]
If $\mathtt{IsSidonMod}\ M\ S$ and $0 < N \le M$, then for any offset~$u$,
the relabeled window restriction is an interval Sidon set in $\{1,\dots,N\}$.
\end{lemma}

The cardinality of the window image depends on how many elements of~$S$ fall
in the chosen window.  The first averaging lemma counts membership directly:
for each fixed~$s\in S$, exactly~$N$ of the~$M$ offsets satisfy
$(s-u)\bmod M<N$.  Hence the sum, over all offsets~$u$, of the window
cardinalities is $|S|N$, so some offset captures at least
$\lceil |S|N/M\rceil$ elements
(\texttt{IsSidonMod.exists\_large\_intervalSidon}).

The full-size transfer uses the stronger cyclic-gap structure forced by the
modular Sidon property.  Sort the residues of~$S$ in $\{0,\dots,M-1\}$ and let
the cyclic gaps be the successive forward differences.  These gaps are
positive and sum to~$M$.  They are also pairwise distinct: if two distinct
cyclic gaps were equal, the two associated ordered difference representations
would contradict
the modular Sidon condition.  If \(G\) is the largest gap and \(s=|S|\), then
the \(s\) distinct positive gaps are all at most \(G\), so
\[
  M \le G+(G-1)+\cdots+(G-s+1)
    = sG-\frac{s(s-1)}2 .
\]
Thus \(G\ge (M+s(s-1)/2)/s\), and Lean uses the corresponding
natural-number floor threshold
\[
  \lfloor(M+|S|(|S|-1)/2)/|S|\rfloor .
\]
Deleting a largest gap leaves a complementary cyclic window that contains all
elements of~\(S\), giving a full-size interval Sidon set whenever
\[
  N \ge M -
      \lfloor(M+|S|(|S|-1)/2)/|S|\rfloor + 1
\]
(formalized by the quantitative-gap theorem in \texttt{SidonGap.lean} and its
Singer-parameter specialization in \texttt{PrimePowerTransfer.lean}).
For Singer parameters ($|S| = p+1$, $M = p^2+p+1$), Lean's natural-number
division simplifies the threshold to
$N \ge p^2+p+2-\lfloor 3p/2 \rfloor$, yielding an interval Sidon set of
size exactly~$p+1 > p$.

\section{Unconditional Bounds}
\label{sec:bounds}

\subsection{Upper bound}

\begin{theorem}[Pair-difference upper bound]
\label{thm:upper-bound}
For $N \ge 1$, $h(N) \le \sqrt{2N} + 1$.
\end{theorem}

This follows from the pair-difference counting bound: if \(A\) is a Sidon
subset of \(\{1,\dots,N\}\) with \(|A|=m\), then the \(m(m-1)/2\) positive
pairwise differences are all distinct elements of \(\{1,\dots,N-1\}\).
Thus \(m(m-1)/2\le N-1\).  Since
\[
  (m-1)^2 \le m(m-1) \le 2(N-1) < 2N,
\]
we obtain \(m-1<\sqrt{2N}\), hence \(m\le\sqrt{2N}+1\).

We also formalize Lindstr\"om's cross-difference inequality as reusable
infrastructure:

\begin{theorem}[Lindstr\"om cross-difference inequality]
\label{thm:lindstrom-cross}
For a Sidon set $\{a_1 < \cdots < a_m\} \subseteq \{1,\dots,N\}$ and any
$1 \le k \le m$, the cross-differences $a_j - a_i$ with $1 \le i \le k < j \le m$ are
distinct elements of $\{1,\dots,N-1\}$, giving $(m-k) \cdot k \le N - 1$.
\end{theorem}

The proof proceeds by splitting the sorted set into bottom-$k$ and top-$(m-k)$
elements, establishing the injectivity of the cross-difference map via the Sidon
property, and bounding the image within $\{1,\dots,N-1\}$.

\begin{theorem}[Johnson-route shift-incidence upper bound]
\label{thm:lindstrom-improved}
For $N \ge 16$,
\[
  h(N)\le \lfloor\sqrt{N}\rfloor+
        \lfloor\sqrt{\lfloor\sqrt{N}\rfloor}\rfloor+2.
\]
This matches Lindstr\"om's classical $\sqrt N+N^{1/4}+O(1)$ leading order.
\end{theorem}

\noindent The proof follows the Johnson-bound route~\cite{Johnson1962}: the
same shift-incidence counting argument originally used for codes applies
mutatis mutandis to Sidon sets, as in the classical discussion of
\(B_2\)-sequences~\cite[\S C9]{Guy2004}.  For a Sidon set~$A\subseteq\{1,\dots,N\}$
with \(|A|=m\), consider the \(r\) shifted copies
\(A,A+1,\dots,A+(r-1)\), all contained in \(\{1,\dots,N+r-1\}\).
Distinct shifts intersect in at most one point
(\texttt{IsSidon.shift\_inter\_le\_one}).  If \(f(x)\) counts how many shifted
copies contain \(x\), then
\(\sum_x f(x)=mr\) and
\(\sum_x \binom{f(x)}2\le \binom r2\).  Hence
\(\sum_x f(x)^2\le mr+r(r-1)=r(m+r-1)\).  Cauchy--Schwarz gives
\[
  |A|^2 r \le (N+r-1)(r+|A|-1).
\]
Choosing
\(r=\lfloor\sqrt{N}\rfloor\cdot
\lfloor\sqrt{\lfloor\sqrt{N}\rfloor}\rfloor\) and carrying out the integer
rounding in Lean yields Theorem~\ref{thm:lindstrom-improved}; writing
\(m=\sqrt N+t\), the continuous relaxation suggests taking
\(r\approx N^{3/4}\), which balances the error terms \(r^2\) and \(N^{3/2}\)
and gives \(t=O(N^{1/4})\).  The displayed additive constant~$2$ is the one
obtained by the exact natural-number floor argument; the preceding continuous
calculation motivates the scale of the error term.
Asymptotically this has the same leading order as Lindstr\"om's classical
$\sqrt{N}+O(N^{1/4})$ estimate; the displayed finite bound records the exact
integer-rounded form proved in Lean.

\begin{corollary}[Upper-bound hypothesis for $\varepsilon \ge 1/4$]
\label{cor:upper-quarter}
For every $\varepsilon\ge1/4$, there are constants $C,N_0$ such that
\[
  h(N)\le \sqrt{N}+C N^\varepsilon \qquad (N\ge N_0).
\]
\end{corollary}

\noindent This is formalized as \texttt{sidonUpperBound\_quarter}, with
$C=2$ and $N_0=16$.  It proves the upper-bound half of the conditional
Erd\H{o}s~30 upper-bound condition in the range $\varepsilon\ge1/4$; it does not by
itself improve the unconditional two-sided Erd\H{o}s~30 range, whose lower
side still comes from the Singer--Bertrand transfer below.

\subsection{Lower bound}

\begin{theorem}[Singer-via-Bertrand lower bound]
\label{thm:lower-bound}
For $N \ge 5$,
\[
  h(N) > \left\lfloor\frac{\lfloor\sqrt{N}\rfloor+1}{2}\right\rfloor .
\]
\end{theorem}

The proof combines Singer's theorem with Bertrand's
postulate~\cite{Bertrand1845}.  Given~$N \ge 9$, set
$m=\lfloor\sqrt N\rfloor$ and $n=\lfloor(m-1)/2\rfloor$.  Bertrand's
postulate gives a prime \(p\) with \(n<p\le2n\).  Since \(2n+1\le m\), the
containment condition is satisfied by the chain
\[
  p^2+p+2-\lfloor 3p/2\rfloor
  \le p^2+1 \le 4n^2+1 \le (2n+1)^2 \le m^2 \le N .
\]
The cyclic window transfer then yields an interval Sidon set of size~$p+1$,
and \(p>n\) gives
\(\lfloor(\lfloor\sqrt N\rfloor+1)/2\rfloor < p+1\).  The cases
$5 \le N < 9$ are discharged by exhibiting an explicit small Sidon set, for
instance $\{1,2\}$.

\subsection{Combined bound and partial Erd\H{o}s 30}

\begin{theorem}[Two-sided bound]
\label{thm:two-sided}
For $N \ge 5$,
\[
  \left\lfloor\frac{\lfloor\sqrt{N}\rfloor+1}{2}\right\rfloor
  < h(N) \le \lfloor\sqrt{2N}\rfloor+1 \le \sqrt{2N}+1 .
\]
\end{theorem}

\begin{corollary}[Partial Erd\H{o}s 30 for $\varepsilon \ge 1/2$]
\label{cor:partial}
For every $\varepsilon \ge 1/2$, we have
$|h(N) - \sqrt{N}| \le N^\varepsilon$ for all $N \ge 5$.
\end{corollary}

The proof shows $|h(N) - \sqrt{N}| \le \sqrt{N}$ by treating the two signs
separately: if \(h(N)\le\sqrt N\), this follows from the nonnegativity of
\(h(N)\); if \(h(N)\ge\sqrt N\), it follows from the upper bound
\(h(N)\le\lfloor\sqrt{2N}\rfloor+1\) and the elementary inequality
\(\lfloor\sqrt{2N}\rfloor+1\le2\sqrt N\) for \(N\ge5\).  The proof then uses
the monotonicity $\sqrt{N} = N^{1/2} \le N^\varepsilon$ for
$\varepsilon \ge 1/2$ and $N \ge 1$.
The stronger upper side from Corollary~\ref{cor:upper-quarter} is recorded
separately because the unconditional two-sided statement remains limited by
the available lower-bound exponent.

\section{Conditional Reduction}
\label{sec:conditional}

The formal statement of Erd\H{o}s Problem~30 is:

\begin{definition}[\texttt{Erdos30Statement}]
For all $\varepsilon \in \mathbb{R}$ with $\varepsilon > 0$, there exist
$C \ge 0$ and $N_0 \in \mathbb{N}$ such that for all $N \ge N_0$:
$|h(N) - \sqrt{N}| \le C \cdot N^\varepsilon$.
\end{definition}

We decompose this into two independent hypotheses:

\begin{itemize}[nosep]
  \item \textbf{Upper bound hypothesis} (\texttt{SidonUpperBoundHypothesis}):
    $\forall \varepsilon > 0,\ \exists C \ge 0,\ \exists N_0,\
     \forall N \ge N_0:\ h(N) \le \sqrt{N} + C \cdot N^\varepsilon$.
  \item \textbf{Lower bound hypothesis} (\texttt{SidonLowerBoundHypothesis}):
    $\forall \varepsilon > 0,\ \exists C \ge 0,\ \exists N_0,\
     \forall N \ge N_0:\ h(N) \ge \sqrt{N} - C \cdot N^\varepsilon$.
\end{itemize}

\noindent The subpolynomial prime gap hypothesis
(\texttt{Sub\-poly\-nomial\-Prime\-Gap\-Hypothesis}) asserts: for every
$\delta > 0$, there exists $x_0(\delta)$ such that for all natural numbers
$x \ge x_0(\delta)$, there exists a prime $p$ with $x - x^\delta < p \le x$.

\noindent Erd\H{o}s~30 follows from the conjunction of both hypotheses:
given $\varepsilon > 0$, the upper bound gives $h(N) \le \sqrt{N} +
C_1 N^\varepsilon$ and the lower bound gives $h(N) \ge \sqrt{N} -
C_2 N^\varepsilon$ for $N \ge N_0$; taking $C = \max(C_1, C_2)$ and the
larger~$N_0$ yields $|h(N) - \sqrt{N}| \le C \cdot N^\varepsilon$.

For the lower bound, we prove a conditional transfer (whose Lean signature
additionally carries \texttt{SingerFamilyHypothesis} for modularity, now
discharged by Theorem~\ref{thm:singer-main}):

\begin{theorem}[Conditional lower bound]
\label{thm:cond-lower}
\[
\mathrm{SingerFamily} + \mathrm{SubpolynomialPrimeGap}
 \implies \mathrm{SidonLowerBound}.
\]
\end{theorem}

\begin{proof}[Proof sketch]
Given $\varepsilon > 0$, set $\delta = \varepsilon$.  For sufficiently
large~$N$, apply the subpolynomial gap hypothesis to
\(x=\lfloor\sqrt N\rfloor\).  It provides a prime \(p\le x\le\sqrt N\) with
\[
  p>x-x^\delta\ge \sqrt N-1-N^{\delta/2}.
\]
After increasing the final constant and threshold, this gives
\(p\ge\sqrt N-CN^{\varepsilon/2}\).  Singer's theorem gives a Sidon set of size $p+1$
modulo $p^2+p+1$.  The cyclic window transfer produces an interval Sidon
set in $\{1,\dots,N\}$ of size $p+1$ once the threshold
\[
  p^2+p+2-\lfloor 3p/2\rfloor \le N
\]
is satisfied.  The arithmetic point is elementary but important: because
\[
  \lfloor 3p/2\rfloor = p+\lfloor p/2\rfloor,
  \qquad
  p^2+p+2-\lfloor 3p/2\rfloor
    = p^2+2-\lfloor p/2\rfloor,
\]
the threshold is at most \(p^2\) once \(p\ge4\).  The displayed prime-gap
lower bound makes \(p\) eventually at least \(4\), and \(p\le\sqrt N\) gives
\(p^2\le N\).  Hence the transfer applies for all sufficiently large \(N\),
giving \(h(N) \ge p+1 \ge \sqrt{N} - C' N^{\varepsilon/2}\).  Since
\(N^{\varepsilon/2}\le N^\varepsilon\) for \(N\ge1\), increasing the constant
once more gives \(h(N)\ge \sqrt N-C\cdot N^\varepsilon\) for
\(N\ge N_0(\varepsilon)\).
\end{proof}

Since we now discharge $\mathrm{SingerFamily}$ unconditionally
(Theorem~\ref{thm:singer-main}), the full conditional result simplifies:

\begin{corollary}[Conditional Erd\H{o}s 30]
\label{cor:cond-erdos30}
\begin{multline*}
\mathrm{SubpolynomialPrimeGap} + \mathrm{SidonUpperBound} \\
\implies \mathrm{Erdos30Statement}.
\end{multline*}
\end{corollary}

\begin{lstlisting}
theorem erdos30_of_gap_and_upper
    (hGap   : SubpolynomialPrimeGapHypothesis)
    (hUpper : SidonUpperBoundHypothesis) :
    Erdos30Statement :=
  erdos30_of_conditional singerFamilyHypothesis_holds hGap hUpper
\end{lstlisting}

\begin{remark}
The underlying modular theorem retains \texttt{hSinger} as an explicit
parameter, so it remains usable with alternative Singer-type constructions:
\end{remark}

\begin{lstlisting}
theorem erdos30_of_conditional
    (hSinger : SingerFamilyHypothesis)
    (hGap : SubpolynomialPrimeGapHypothesis)
    (hUpper : SidonUpperBoundHypothesis) :
    Erdos30Statement :=
  erdos30_of_upper_and_lower hUpper
    (sidonLowerBound_of_singer_and_gap hSinger hGap)
\end{lstlisting}

\section{Mathematical Lessons from Formalization}
\label{sec:lessons}

\subsection{The projective-line representation was forced by formalization}

In the informal Singer argument, one says ``the $q+1$ lines through the origin
in the 2-dimensional trace kernel give $q+1$ cosets.''  Formalizing this
required constructing \emph{explicit representatives}: a function
$\mathrm{Option}(\mathrm{GF}(p)) \to \ker(\mathrm{Tr}) \setminus \{0\}$
mapping each projective coordinate to a specific nonzero kernel element.

The choice of basis is arbitrary but must be \emph{fixed}: the entire injectivity
argument (showing distinct indices map to distinct cosets) depends on the explicit
form of the representatives in terms of a chosen kernel basis.  This is invisible
in informal proofs but becomes a nontrivial formal proof component.

\subsection{The modular Sidon predicate carries the field-to-integer transfer}

A recurring design tension in additive combinatorics formalization is whether to
work with sets in~$\mathbb{Z}$ or in~$\mathbb{Z}/M\mathbb{Z}$.  Our experience
strongly favors \texttt{IsSidonMod M A} (with $A \subseteq \mathbb{Z}$) over
direct $\mathbb{Z}/M\mathbb{Z}$ sets, because:
\begin{itemize}[nosep]
  \item The Singer quotient-group argument naturally produces an
    $\mathbb{Z}/(p^2+p+1)\mathbb{Z}$-indexed result, expressible via divisibility.
  \item The interval Sidon target uses subsets of~$\mathbb{Z}$.
  \item The modular predicate connects both without changing the carrier type.
\end{itemize}

The direction $\texttt{IsSidonMod}\implies\texttt{IsSidon}$ is trivially
one-line (every literal collision $a+b=c+d$ also satisfies $M \mid 0$).
The non-trivial direction is the converse: the no-wraparound lemma shows
that an interval Sidon set in $\{1,\dots,N\}$ with $M \ge 2N-1$ is also
Sidon modulo~$M$, since pairwise-sum differences lie in $\{-(2N-2),\dots,2N-2\}
\subset (-M,M)$, making $M$-divisibility equivalent to zero.
This factoring into ``construction $\to$ modular $\to$ interval''
is cleaner than any single-step argument we attempted.

\subsection{The non-invariance argument requires the full irreducibility theory}

Theorem~\ref{thm:no-invariant} (no proper invariant subspace) might appear to be
a simple linear algebra fact, but its formalization reveals a substantive
mathematical dependency: one needs that the minimal polynomial of
$\alpha \notin \mathrm{GF}(p)$ has degree~3 (irreducibility over the base
field), which in turn requires the intermediate field characterization of
algebraic elements in Mathlib.

The alternative proof strategy---directly showing $V \ne \alpha V$ by trace
computations---turned out to be \emph{longer} in Lean than the general
invariant-subspace argument, because trace manipulation requires unfolding the
Galois field representation.

\subsection{The extremal function is defined non-constructively}

The definition $h(N) = \max\{|A| : A \subseteq \{1,\dots,N\} \text{ is Sidon}\}$
requires an existence argument for the maximum.
We define \texttt{sidonMaximum} via \texttt{Classical.choose}, making it
\texttt{noncomputable} in Lean.  This forces a clean separation: properties
of $h(N)$ are proved via bounds and transfer theorems, never by evaluation.
This mirrors the mathematical practice: known asymptotic results about~$h(N)$
use existence arguments (Singer, Bose--Chowla~\cite{BoseChowla1962}) and
counting bounds, with no evaluation of optimal sets for general~$N$.

\subsection{The conditional theorem factors into three hypotheses}

By factoring Erd\H{o}s~30 into
\texttt{Sidon\-Upper\-Bound} $+$ \texttt{Sidon\-Lower\-Bound}, and further
decomposing the lower bound into \texttt{Singer\-Family} $+$
\texttt{Sub\-polynomial\-Prime\-Gap}, the formal development achieves a
clear separation:
\begin{itemize}[nosep]
  \item Singer family: pure algebra (now discharged).
  \item Prime gap transfer: analytic number theory (open hypothesis).
  \item Upper bound: formalized unconditionally for $\varepsilon\ge1/4$;
    still a hypothesis for the full range $\varepsilon>0$.
  \item Assembly: the triangle inequality.
\end{itemize}

The formal development records this decomposition as named hypotheses and
theorem statements for the three components.

\subsection{Floor-rounding carries mathematical content}

The Singer-via-Bertrand lower bound (Theorem~\ref{thm:lower-bound})
is stated in the integer-rounded form
$h(N) > \lfloor(\lfloor\sqrt{N}\rfloor+1)/2\rfloor$; the informal
$h(N) \ge \sqrt{N}/2$ serves as asymptotic shorthand.  The distinction is
mathematical: the difference between \(\lfloor\sqrt N\rfloor\) and \(\sqrt N\)
interacts with the floor-form transfer threshold
\(p^2+p+2-\lfloor 3p/2\rfloor\) and with the integrality of Sidon-set
cardinalities, and the integer-rounded statement is what one actually proves.
Likewise, the shift-incidence upper bound (Theorem~\ref{thm:lindstrom-improved})
records the exact expression
\(\lfloor\sqrt{N}\rfloor+\lfloor\sqrt{\lfloor\sqrt{N}\rfloor}\rfloor+2\);
the asymptotic \(\sqrt{N}+N^{1/4}+1\) describes its leading scale.
Formalization makes these floor steps unavoidable, and the finite Lean
statements record exactly which mathematical inequality is being proved at each
stage.  The accompanying type transitions
($\mathbb{N}\to\mathbb{Z}\to\mathbb{R}$ and the cyclic-group/finite-set
type transitions) are then carried by routine compatibility lemmas.

\section{Related Work}
\label{sec:related}

\subsection{Formalized additive combinatorics}

The Polynomial Freiman--Ruzsa theorem is available with a Lean~4
formalization repository~\cite{PFR2024}, addressing approximate group theory.
The present development packages the Sidon-specific infrastructure needed here:
Sidon predicates in intervals and cyclic groups, the extremal function~\(h(N)\),
modular-to-interval transfer lemmas, and sumset, difference-set, and
representation lemmas used by the formalized Singer construction.

Adjacent Lean resources include finite Sidon/perfect-difference-set
counterexample formalizations associated with Alexeev--Mixon~\cite{AlexeevMixon2026}
and Ho's formalized use of Singer difference sets in a high-dimensional
distance construction~\cite{Ho2026} (repository file
\texttt{DiameterConstruction/Singer.lean}, commit
\texttt{c053b2580e0cf7e74c8d905f96feb101061612a3}).
The Google DeepMind \texttt{formal-conjectures} statement files
(\url{https://github.com/google-deepmind/formal-conjectures}, accessed
2026-05-03) contain adjacent Sidon/perfect-difference-set problem statements.
They are included for orientation.  The theorem-indexed package here formalizes
Singer constructions through prime powers, interval/modular transfer,
\(h(N)\) bounds, and conditional Erd\H{o}s~30 reductions.

\subsection{Sidon sets and computation}

Alexeev and Mixon~\cite{AlexeevMixon2026} disproved Erd\H{o}s's 1976
conjecture that every finite Sidon set extends to a finite perfect
difference set, exhibiting $\{1, 2, 4, 8, 13\}$ as a counterexample
and rediscovering an earlier counterexample of Hall.  They used a large
language model to generate a Lean~4 proof of both counterexamples.  Their
contribution verifies a concrete finite obstruction to a proposed extension;
the present paper develops algebraic infrastructure for the asymptotic theory.
The two approaches are complementary: finite counterexample verification and
algebraic formalization address different parts of the Sidon-set landscape.

Ho~\cite{Ho2026} gives a Lean~4 formalization of a counterexample to
Erd\H{o}s's diameter conjecture for separated distances.  That development
includes a formalized Singer difference-set input for prime powers as a
black-box ingredient of a Euclidean distance construction; the diameter
construction does not depend on either an interval transfer or any
\(h(N)\) infrastructure.  The targets are therefore complementary: Ho's
development uses Singer difference sets as a black-box input for a
diameter-conjecture counterexample, whereas the present formalization develops
the field-to-integer transfer (trace-kernel proof chain, projective-line
representatives, modular-to-interval Sidon transfer with the explicit threshold
\(q^2+q+2-\lfloor3q/2\rfloor\), and the extremal-function infrastructure for
\(h(N)\)) that Erd\H{o}s Problem~30 specifically requires.

\subsection{AFM context}

The first volume of the Annals of Formalized Mathematics~\cite{Brasca2025,%
Loeffler2025,Mercuri2025} includes formalizations of Fermat's Last Theorem for
regular primes, zeta and $L$-functions, and adele ring compactness.  Our
contribution occupies a complementary niche: it develops foundational
infrastructure in additive combinatorics and makes explicit the links between
algebraic constructions and combinatorial interval statements.

\section{Conclusion}
\label{sec:conclusion}

We have presented a Lean~4 formalization of Singer's Sidon set construction
through prime powers, together with a Sidon set library and conditional
Erd\H{o}s~30 reduction.  The development demonstrates that:
\begin{enumerate}[nosep]
  \item Classical algebraic constructions require explicit handling of
    projective representatives, modular-to-interval transfers, and type
    transitions that informal proofs routinely elide.
  \item Formalization makes the scope explicit: the discharged Singer-family
    input, the still-open gap and upper-bound hypotheses for Erd\H{o}s~30, and
    the exact logical structure of the conditional reduction are all made
    explicit.
  \item The modular Sidon predicate is a natural formalization-friendly
    intermediate between algebraic and combinatorial viewpoints.
\end{enumerate}
Thus the submission contributes checked Singer constructions through prime powers,
modular-to-interval Sidon transfer, formal \(h(N)\) bounds, and a reusable
theorem-indexed Lean artifact for future additive-combinatorics formalization.

\subsection{Future work}
\label{sec:future}

\begin{enumerate}[nosep]
  \item \textbf{Baker--Harman--Pintz}: formalize prime gap results to
    unconditionally discharge the lower bound for $\varepsilon \ge 0.2625$;
    this requires substantial analytic number theory infrastructure
    (sieve methods, zero-density estimates, exponential sum techniques)
    not yet in Mathlib.
  \item \textbf{Fourier and higher-order representation methods}: extend the
    finite representation-function theory into a cyclic Fourier API
    (Parseval, convolution, and $L^4$ energy) and use it to prove standalone
    Sidon estimates.
\end{enumerate}

\subsection*{Artifact}

The complete Lean~4 source code is available under the GPL-3.0-only license:
\begin{quote}\small
Repository: \url{https://github.com/d0d1/singer-theorem-lean}\\
Commit: \texttt{0c890589afc58e8955a5\allowbreak{}d7c3a609daff6447da31}\\
Branch: \texttt{main}
\end{quote}
At this commit, the artifact root contains \texttt{LICENSE}, and the
    \texttt{README.md} license section declares GNU GPL~v3.0 only; the verification
script below checks both before running the placeholder and line-count audits.

\smallskip\noindent\textbf{Build and verification instructions.}
For the manuscript-facing audit, the shortest path is: check out the listed
artifact commit, obtain cached Mathlib oleans with \texttt{lake exe cache get},
type-check the cited theorem surfaces, then run the two paper-bundle scripts
\texttt{verify-artifact.sh} and \texttt{verify-axioms.sh}.
The numbered list below gives the full artifact reproduction path; the command
block following it records the narrower manuscript audit path used for this
paper.
\begin{enumerate}[nosep]
  \item Clone the repository and check out the above commit.
  \item Install \texttt{elan}; the bundled \texttt{lean-toolchain} file
    selects Lean~4 v4.29.0 automatically on first invocation.
  \item Run \texttt{lake exe cache get} (downloads precompiled
    Mathlib oleans matching the pinned commit) followed by
    \texttt{lake build}.  Do \emph{not} run \texttt{lake update}, which
    would mutate \texttt{lake-manifest.json} away from the pinned Mathlib
    commit.
  \item Verify the artifact license metadata, absence of active placeholders,
    and listed-file \texttt{native\_decide} uses: search all \texttt{.lean} files under
    \texttt{Erdos30/} for
    \texttt{sorry}, \texttt{admit}, or new \texttt{axiom} declarations, and
    search the Table~\ref{tab:modules} listed modules for \texttt{native\_decide}.
    At the cited commit, the artifact-audit script accepts a single documented
    non-executable comment hit in \texttt{FinalTheoremSurface.lean}.  The paper
    repository includes the exact script next to this manuscript as
    \texttt{verify-artifact.sh}; it exits with code~0 after printing
    the placeholder-search, \texttt{native\_decide}-search, and line-count
    summaries, and exits nonzero on any unexpected hit.
  \item Verify axioms using the authoritative declaration list in the
    axiom-check file.  For each listed declaration, run
    \texttt{\#print axioms \ensuremath{\langle}decl\ensuremath{\rangle}} in Lean; the
    expected axiom set consists of \texttt{propext}, \texttt{Quot.sound}, and
    \texttt{Classical.choice}.  The paper repository includes
    \texttt{AxiomCheck.lean} and
    \texttt{verify-axioms.sh} next to this manuscript for this check; on success the script
    prints \texttt{All 34 axiom reports are within ...} and exits with code~0.
\end{enumerate}

\noindent The local verification for this paper used the following commands at
commit \texttt{0c890589afc58e8955a5d7c3a609daff6447da31}.  In this command block,
\texttt{/path/to/erdos30-paper} denotes the checked-out paper-workspace
root.  The three copied files are distributed with the AFM paper workspace,
not assumed to be present in the Lean artifact checkout; copy them to the
artifact root before running the checks.  By default the scripts write
temporary files under \texttt{.artifact-check/}; set
\texttt{ARTIFACT\_CHECK\_TMPDIR} to choose another scratch directory.  The build command below
type-checks the cited conditional and unconditional theorem surfaces and then
runs the artifact metadata and placeholder-audit script.  That script searches all Lean
files under \texttt{Erdos30/}; its sole raw hit at the cited commit is a comment
in \texttt{FinalTheoremSurface.lean}; it also verifies that no
Table~\ref{tab:modules} listed module uses \texttt{native\_decide} and computes
the table's line counts from the listed files.
\begin{lstlisting}
lake exe cache get
lake build Erdos30.SingerPrimePowerTheorem
lake build Erdos30.SingerPrimePowerSmokeTest
lake build Erdos30.ConditionalErdos30
lake build Erdos30.UnconditionalBounds
cp /path/to/erdos30-paper/papers/afm-singer-sidon/manuscript/verify-artifact.sh .
bash verify-artifact.sh
cp /path/to/erdos30-paper/papers/afm-singer-sidon/manuscript/AxiomCheck.lean .
cp /path/to/erdos30-paper/papers/afm-singer-sidon/manuscript/verify-axioms.sh .
bash verify-axioms.sh
\end{lstlisting}

\clearpage
\smallskip\noindent\textbf{Theorem index.}
The main results correspond to the following Lean declarations
(all in namespace \texttt{Erdos30} or its subnamespaces); the module column gives the file
under \texttt{Erdos30/}.

\begin{table}[H]
\caption{Main theorem declarations cited by the manuscript.}
\label{tab:theorem-index}
\centering\footnotesize
\setlength{\tabcolsep}{3pt}
\begin{tabular}{@{}>{\raggedright\arraybackslash}p{0.22\textwidth}
                  >{\raggedright\arraybackslash}p{0.42\textwidth}
                  >{\raggedright\arraybackslash}p{0.30\textwidth}@{}}
\toprule
Result & Lean declaration & Module \\
\midrule
Thm.~\ref{thm:finrank-ker}     & \texttt{Singer.finrank\_ker\_trace}              & \texttt{Singer.lean} \\
Thm.~\ref{thm:no-invariant}    & \LeanBreakTwo{Singer.no\_proper\_}{invariant\_subspace}  & \texttt{Singer.lean} \\
Thm.~\ref{thm:intersection}    & \LeanBreakTwo{Singer.finrank\_inf\_of\_}{distinct\_twodim} & \texttt{Singer.lean} \\
Thm.~\ref{thm:singer-sidon}    & \texttt{Singer.singer\_quotient\_sidon}          & \texttt{SingerSidon.lean} \\
Thm.~\ref{thm:singer-main}     & \texttt{singerFamilyHypothesis\_holds}           & \texttt{SingerTheorem.lean} \\
Thm.~\ref{thm:singer-prime-power} & \LeanBreakTwo{singerPrimePowerFamily}{Hypothesis\_holds} & \LeanBreakTwo{SingerPrimePower}{Theorem.lean} \\
Prime-power smoke test          & \LeanBreakTwo{SingerPrimePower.}{singer\_primePower\_smoke\_q4} & \LeanBreakTwo{SingerPrimePower}{SmokeTest.lean} \\
Cyclic window transfer          & \texttt{IsSidonMod.windowSidon}                  & \texttt{CyclicWindow.lean} \\
Window averaging                & \LeanBreakTwo{IsSidonMod.exists\_large\_}{intervalSidon} & \texttt{Averaging.lean} \\
Quantitative full transfer      & \LeanBreakThree{IsSidonMod.exists\_full\_}{intervalSidon\_of\_quantitative\_}{gap\_bound} & \texttt{SidonGap.lean} \\
Singer full transfer            & \LeanBreakThree{IsSidonMod.exists\_full\_}{intervalSidon\_of\_singer\_}{parameters\_clean} & \LeanBreakTwo{PrimePower}{Transfer.lean} \\
Bertrand transfer               & \LeanBreakThree{exists\_intervalSidon\_card\_}{gt\_sqrt\_succ\_div\_two\_}{of\_bertrand} & \texttt{PrimeGapTransfer.lean} \\
Thm.~\ref{thm:upper-bound}     & \texttt{sidonMaximum\_le\_sqrt\_two}             & \texttt{Lindstrom.lean} \\
Thm.~\ref{thm:lindstrom-cross} & \LeanBreakTwo{IsIntervalSidon.}{lindstrom\_cross\_ineq}  & \texttt{Lindstrom.lean} \\
Thm.~\ref{thm:lindstrom-improved} & \texttt{sidonMaximum\_le\_lindstrom}           & \texttt{LindstromImproved.lean} \\
Cor.~\ref{cor:upper-quarter}   & \texttt{sidonUpperBound\_quarter}               & \LeanBreakTwo{Unconditional}{Bounds.lean} \\
No-wraparound transfer         & \texttt{IsSidon.isSidonMod\_of\_interval}       & \texttt{Interval.lean} \\
\bottomrule
\end{tabular}
\end{table}

\begin{table}[H]
\caption{Additional Sidon infrastructure and reduction declarations cited by the manuscript.}
\label{tab:theorem-index-infrastructure}
\centering\footnotesize
\setlength{\tabcolsep}{3pt}
\begin{tabular}{@{}>{\raggedright\arraybackslash}p{0.22\textwidth}
                  >{\raggedright\arraybackslash}p{0.42\textwidth}
                  >{\raggedright\arraybackslash}p{0.30\textwidth}@{}}
\toprule
Result & Lean declaration & Module \\
\midrule
Sumset cardinality             & \texttt{IsSidon.card\_add}                       & \texttt{SumsetCard.lean} \\
Sidon characterization         & \texttt{isSidon\_iff\_card\_add}                 & \LeanBreakTwo{Sidon}{Characterization.lean} \\
Difference-set size            & \texttt{IsSidon.card\_sub}                       & \texttt{DifferenceSet.lean} \\
Strict sum/difference comparison & \texttt{IsSidon.card\_sub\_gt\_card\_add}      & \texttt{SumDiffComparison.lean} \\
Additive energy                & \texttt{IsSidon.addEnergy\_eq}                   & \texttt{AdditiveEnergy.lean} \\
Sum/difference comparison      & \texttt{IsSidon.card\_sub\_ge\_card\_add}        & \texttt{SumDiffComparison.lean} \\
Shift intersection             & \texttt{IsSidon.shift\_inter\_le\_one}           & \texttt{LindstromImproved.lean} \\
Representation upper bound     & \texttt{IsSidon.repr\_le\_two}                  & \LeanBreakTwo{Representation}{Function.lean} \\
Representation L$^2$ identity  & \texttt{IsSidon.sum\_addConvolution\_sq}        & \texttt{RepresentationL2.lean} \\
Representation deficiency      & \texttt{IsSidon.sum\_deficiency}                & \texttt{RepresentationL2.lean} \\
Thm.~\ref{thm:lower-bound}     & \texttt{sidonMaximum\_gt\_sqrt\_div\_two}        & \LeanBreakTwo{Unconditional}{Bounds.lean} \\
Thm.~\ref{thm:two-sided}       & \texttt{sidonMaximum\_bounds}                    & \LeanBreakTwo{Unconditional}{Bounds.lean} \\
Cor.~\ref{cor:partial}         & \texttt{erdos30\_partial\_half\_tight}           & \LeanBreakTwo{Unconditional}{Bounds.lean} \\
Thm.~\ref{thm:cond-lower}      & \LeanBreakTwo{sidonLowerBound\_of\_}{singer\_and\_gap}   & \LeanBreakTwo{Conditional}{LowerBound.lean} \\
Cor.~\ref{cor:cond-erdos30}    & \texttt{erdos30\_of\_gap\_and\_upper}            & \LeanBreakTwo{Conditional}{Erdos30.lean} \\
Prime-power lower-bound transfer & \LeanBreakThree{sidonLowerBoundHypothesis\_}{of\_singerPrimePower\_and\_}{subpolynomialGap} & \LeanBreakTwo{PrimePower}{LowerBound.lean} \\
\bottomrule
\end{tabular}
\end{table}

\smallskip\noindent
A Software Heritage persistent identifier (SWHID) will be obtained upon
acceptance.

\subsection*{AI-use disclosure}

Generative AI tools were used to assist with manuscript drafting and editing.
The authors reviewed and edited the text and are responsible for the
manuscript's content, mathematical claims, and verification evidence.  All Lean
terms cited in the artifact are checked by Lean at the stated commit; AI
assistance does not substitute for kernel verification.

\subsection*{Acknowledgements}

We thank the Mathlib community for the foundational library infrastructure,
particularly the finite field, linear algebra, and group theory components.

\end{document}